\documentclass [12pt, A4]{amsart}   

\usepackage{comment}

\usepackage{a4wide}    

\usepackage[utf8]{inputenc}   
\usepackage[T1]{fontenc}            
\usepackage[english, french, greek, english]{babel}       
\usepackage{textcomp}                
\usepackage{graphicx}       
\usepackage{listing}        
\usepackage{geometry}     

\usepackage{epigraph}
\usepackage{amsthm}
\usepackage{amsmath}
\usepackage{amssymb}
\usepackage{amsfonts}
\usepackage{mathrsfs}
\usepackage{dsfont}
\usepackage{mathabx}
\usepackage{envmath}
\usepackage{cases}

\usepackage{epic, ecltree}
\usepackage{epsfig}
\usepackage{graphics}
\usepackage{color}

\usepackage{pgf,tikz,pgfplots} 
\usetikzlibrary{patterns}

\usepackage{mathrsfs}
\usetikzlibrary{arrows}

\makeatletter
\pgfdeclarepatternformonly[\LineSpace]{my north east lines}{\pgfqpoint{-1pt}{-1pt}}{\pgfpoint{\LineSpace+0.2pt}{\LineSpace+0.2pt}}{\pgfqpoint{\LineSpace}{\LineSpace}}%
{
  \pgfsetcolor{\tikz@pattern@color}
  \pgfsetlinewidth{0.4pt}
  \pgfpathmoveto{\pgfqpoint{0pt}{0pt}}
  \pgfpathlineto{\pgfqpoint{\LineSpace + 0.1pt}{\LineSpace + 0.1pt}}
  \pgfusepath{stroke}
}
\pgfdeclarepatternformonly[\LineSpace]{my north west lines}{\pgfqpoint{-1pt}{-1pt}}
{\pgfqpoint{\LineSpace}{\LineSpace}}{\pgfqpoint{\LineSpace}{\LineSpace}}%
{
  \pgfsetcolor{\tikz@pattern@color}
  \pgfsetlinewidth{0.4pt}
  \pgfpathmoveto{\pgfqpoint{0pt}{\LineSpace}}
  \pgfpathlineto{\pgfqpoint{\LineSpace + 0.1pt}{-0.1pt}}
  \pgfusepath{stroke}
}

\newdimen\LineSpace
\tikzset{
  line space/.code={\LineSpace=#1},
  line space=3pt
}
\makeatother

\usepackage{newunicodechar}
\newunicodechar{·}{\kern-.25em\textperiodcentered\kern-.25em}

\usepackage{fancyhdr}     

\usepackage{lmodern}
\DeclareFontFamily{OMX}{lmex}{}
\DeclareFontShape{OMX}{lmex}{m}{n}{<->lmex10}{}

\usepackage[colorlinks=true, linkcolor=black, urlcolor=blue, filecolor=black, citecolor=black, menucolor=black]{hyperref}  

\usepackage{enumitem} 

\newcommand{\enstq}[2]{\left\lbrace#1\mathrel{}\middle|\mathrel{}#2\right\rbrace} 

\newcommand{\ensemblenombre}[1]{\mathbb{#1}} 
\newcommand{\N}{\ensemblenombre{N}}
\newcommand{\Z}{\ensemblenombre{Z}}

\newcommand{\R}{\ensemblenombre{R}}
\newcommand{\C}{\ensemblenombre{C}}
\newcommand{\T}{\ensemblenombre{T}}

\newcommand{\Hol}{\operatorname{Hol}}
\newcommand{\Ran}{\operatorname{Ran}}

\newcommand{\dist}{\operatorname{dist}}

\renewcommand{\Re}{\operatorname{Re}}
\renewcommand{\Im}{\operatorname{Im}}

\newcommand{\prodscal}[2]{\left\langle#1,#2\right\rangle} 

\makeatletter
\newcommand{\bigint}{\@ifnextchar_\@bigintsub\@bigintnosub}
\def\@bigintsub_#1{\def\@int@subscript{#1}\@ifnextchar^\@bigintsubsup\@bigintsubnosup}
\def\@bigintsubsup^#1{\mathop{\text{\large$\int_{\text{\normalsize$\scriptstyle\@int@subscript$}}^{\text{\normalsize$\scriptstyle#1$}}$}}\nolimits}
\def\@bigintsubnosup{\mathop{\text{\large$\int_{\text{\normalsize$\scriptstyle\@int@subscript$}}$}}\nolimits}
\def\@bigintnosub{\@ifnextchar^\@bigintnosubsup\@bigintnosubnosup}
\def\@bigintnosubsup^#1{\mathop{\text{\large$\int^{\text{\normalsize$\scriptstyle#1$}}$}}\nolimits}
\def\@bigintnosubnosup{\mathop{\text{\large$\int$}}\nolimits}
\makeatother

\newtheorem{thm}{Theorem}
\newtheorem{prop}[thm]{Proposition}
\newtheorem{cor}[thm]{Corollary }
\newtheorem{lem}[thm]{Lemma}

\theoremstyle{definition} 

\newtheorem{rmk}[thm]{Remark}

\theoremstyle{remark} 

\title{Reachable space of the Hermite heat equation with boundary control}
\author{Andreas \bsc{Hartmann}} 
\author{Marcu-Antone \bsc{Orsoni}}
\address{Andreas \bsc{Hartmann}, Université de Bordeaux, CNRS, Bordeaux INP, IMB, UMR 5251, 351 Cours de la Libération, F-33400, Talence, France, andreas.hartmann@math.u-bordeaux.fr}
\address{Marcu-Antone \bsc{Orsoni}, Université de Bordeaux, CNRS, Bordeaux INP, IMB, UMR 5251, 351 Cours de la Libération, F-33400, Talence, France, orsoni.ma@gmail.com}
\date{\today}


\usepackage{remreset}
\makeatletter \@removefromreset{figure}{subsection}\makeatother
\makeatletter \@removefromreset{figure}{section}\makeatother

\begin{document}
\maketitle

\begin{abstract}
We discuss reachable states for the Hermite heat equation on a segment with boundary
$L^2$-controls. The Hermite heat equation corresponds to the heat equation to which a quadratic potential is added. We will discuss two situations: when one endpoint of the segment 
is the origin and when the segment is symmetric with respect to the origin. One of the main 
results is that reachable states extend to functions in a Bergman space on a square one diagonal of which is the segment under consideration, and that functions holomorphic in a neighborhood of this square are reachable.
\end{abstract}

\section{Introduction}

An important branch in control theory concerns the identification of reachable states of a given
linear system. In a seminal paper going back to the 70's, Fattorini and Russell \cite{FR1971}
investigated the 
situation for linear parabolic equations in one space dimension with boundary $L^2$-control. 
A special occurrence of such 
equations is the heat equation on a segment which has met a growing interest in recent years.
Let us mention in particular work by Martin, Rosier and Rouchon \cite{MRR2016} which re-launched the
investigation on this problem in 2016 and trying to better understand the role of holomorphic
functions in this setting. Two years later, Dard\'e and Ervedoza \cite{DE18} 
were able to show that functions
holomorphic in a neighborhood of the square $D$ a diagonal of which is the interval on which the
heat equation has to be controlled are reachable 
(previous work already showed that the reachable states
extend holomorphically to that square). In \cite{HKT2020}, the connection with spaces of holomorphic
functions as reachable states was established for the first time conjecturing in particular that
the reachable states should correspond to the Bergman space on $D$. In \cite{Ors2021}, Orsoni 
found that the problem boils down to a so-called separation of singularities problem. 
The final step was achieved in \cite{HO2021} where we solved the separation of singularities problem
thereby establishing the conjecture raised in \cite{HKT2020} as correct: the reachable states of the
1-D heat equation on a segment with boundary $L^2$-controls are given by the Bergman
space on $D$. It should be mentioned that earlier work by Aikawa, Hayashi and Saitoh \cite{AHS1990}
provides a solution to the problem in the case of the half-line. We mention also the paper \cite{KNT2021} which provides a wealth of results on the control of the 1-D heat equation.
\\

It is natural, and in the spirit of the original paper by Fattorini and Russell, to consider  
more general parabolic systems, and in particular to add a potential (this amounts to replace
the Laplace operator by a Sturm-Liouville operator). Clearly, if one expects again holomorphic
functions on $D$ to be reachable,
the least requirement on the potential should be analyticity in $D$.
Very recently, Laurent and Rosier discussed in \cite{LR2020} a result in a rather general setting showing 
that functions holomorphic in a sufficiently big neighborhood of $D$ are reachable.
\\

The aim of this paper is to use techniques from \cite{HO2021} to give a more precise information
in the special case of a quadratic potential. The corresponding equation is also called the
Hermite heat equation. We will consider two cases: a segment with
one endpoint being zero (in this endpoint the equation behaves very much like the classical
heat equation and the corresponding control) and a segment symmetric with respect
to the origin. For the endpoints
different from zero, we get a less precise information (see precise statements below). 
Still it seems fairly reasonnable to expect again the Bergman space on $D$ to characterize the
reachable states also in this situation. A central ingredient in our method is the explicit form of
the kernel --- the Mehler kernel --- for the Hermite heat equation, which in several aspects behaves like the heat kernel. Another central tool comes from the recent discussion in
Strohmaier-Water \cite{SW2020} which gives a general recipe to reachable states in higher
dimension exploring certain symmetry properties of the domains considered.

\subsection{On a segment}
We consider respectively the heat equation and the Hermite heat equation on the segment $[0, \pi]$ with Dirichlet control at both ends which we recall here:
\begin{align}
&\begin{cases}
\label{eq-Heat}
\tag{HE}
\partial_t w  -\partial_x^2w = 0 , \quad t>0, \  x \in (0, \, \pi) \\
w(t, 0)=u_0(t), \qquad w(t, \pi)=u_\pi(t), \quad t>0 \\
w(0, x)=f, \quad x \in (0, \, \pi).
\end{cases}\\
&\begin{cases}
\label{eq-Hermiteheat}
\tag{HHE}
\partial_t w  -\partial_x^2w  +x^2 w = 0 , \quad t>0,\  x \in (0, \, \pi) \\
w(t, 0)=u_0(t), \qquad w(t, \pi)=u_\pi(t), \quad t>0 \\
w(0, x)=f, \quad x \in (0, \, \pi).
\end{cases}
\end{align}

For every control $u=(u_0, \, u_1) \in L^2\left((0, \, \infty), \, \C^2\right)$ and every initial state $f \in X:=W^{-1,2}(0, \pi)$, these equations admit a unique solution $w \in C\big([0, \infty); \, X\big)$ (see e.g. \cite[Theorem 10.8.3]{TW2009}) given respectively by 
$$w(t)=\T_t f + \Phi_t u$$
and 
$$w(t)=\T_t^{\mathrm{H}}f + \Phi^{\mathrm{H}}_t u$$
where $\left(\T_t\right)_{t \ge 0}$ (resp. $\left(\T_t^{\mathrm{H}}\right)_{t \ge 0}$) denotes the Dirichlet Laplacian semigroup (resp. the semigroup generated by the (one dimensional) harmonic oscillator $\mathrm{H}=\partial_x^2 - x^2$) and $\Phi_t$ (resp. $\Phi^{\mathrm{H}}_t$) denotes the \emph{input-to-state map} (see \cite[Proposition 4.2.5]{TW2009}) which is a bounded operator from $L^2\left((0, \, t), \, \C^2\right)$ to $X$. 
Thus, we say that a function $g \in X$ is reachable from $f \in X$ in time $\tau >0$ for \eqref{eq-Heat} (resp. \eqref{eq-Hermiteheat}) if there exists a control $u \in L^2\left((0, \, \tau), \, \C^2\right)$ such that the solution $w$ of \eqref{eq-Heat} (resp. \eqref{eq-Hermiteheat}) satisfies $w(\tau, \cdot)=g$. Due to the smoothing effect of equations \eqref{eq-Heat} and \eqref{eq-Hermiteheat}, it is clear that we cannot reach any non-regular function in $X$. Also, it is well-known that these two equations are null-controllable in any time $\tau>0$, which means that we can reach the null function from any initial condition $f \in X$ and in any time $\tau>0$ (this can be seen for example combining \cite[Proposition 11.5.1]{TW2009} and the proof of \cite[Proposition 11.5.4]{TW2009} made for the heat equation). Clearly, this means that their respective sets of  reachable states do not depend on the initial condition $f \in X$ and are equal respectively to the linear spaces $\Ran \Phi_\tau$ and $\Ran \Phi^{\mathrm{H}}_\tau$. As noticed by Seidman \cite{Seidman1979} (see also \cite[Remark 1.1]{HKT2020}), this also implies that they do not depend on time $\tau >0$ neither. Therefore we will call them \emph{reachable spaces} of equations \eqref{eq-Heat} and \eqref{eq-Hermiteheat}.

As already mentioned, the problem of describing the reachable space for the heat equation \eqref{eq-Heat} has a long history which started with the work of Fattorini and Russell \cite{FR1971}. After a series of developments in the eighties (see for example \cite{Schmidt86} and the references therein), it has then known a renewed interest very recently in \cite{MRR2016}, \cite{DE18}, \cite{HKT2020}, \cite{Ors2021}, \cite{KNT2021} culminating with an exact and definitive characterization in \cite{HO2021}. Solving a problem of separation of singularities, we proved in this last paper that the reachable space of the heat equation \eqref{eq-Heat} is the Bergman space $A^2(D)$ on the square 
$$ D= \enstq{z=x+iy \in \C}{|x-\frac{\pi}{2}| + |y| < \frac{\pi}{2}}.$$
We recall that for a non-negative measurable function $\omega$, 
the Bergman space $A^2(\Omega, \omega)$ on a domain $\Omega$ of the complex plane consists of all the functions $f$ which are holomorphic on $\Omega$ and satisfy 
$$\int_\Omega |f(x+iy)|^2  \omega(x+iy)dx dy < \infty.$$
When $\omega = 1$, we simply write $A^2(\Omega)$. 
\\

The question of describing the reachable space can be generalized to the $n$-dimen\-sional setting considering a bounded open set $O$ instead of a segment. Very recently Strohmaier and Waters have made in \cite{SW2020} a big step forward giving an optimal result about the domain of holomorphy of the reachable states when $O$ is a ball $B:=B(0, \, R)$. More precisely, they proved that the reachable states are holomorphic on the cone given by 
$$\mathcal{E}(B)= \enstq{z=x+iy \in \C^n}{x \in B, \ |y|<\dist(x, \partial B)}$$
and that the functions which are holomorphic on a larger cone $\mathcal{E}(B(0, R'))$ with $R'>R$ are reachable. 
Notice that the domain $\mathcal{E}(B)$ coincides with the square whose one diagonal is $(-R, \, R)$ when $n=1$ (i.e. $B=(-R, \, R)$). \\

In this connection we also mention recent activities on reachable spaces obtained with smooth controls as investigated in \cite{MRR2016} (for control of Gevrey order $2$)  and \cite{KNT2021} (for controls belonging to Sobolev spaces). 
\\

The first aim of this paper is to discuss the case of the Hermite heat equation \eqref{eq-Hermiteheat}. Indeed, some of the tools developed in \cite{HKT2020},\cite{Ors2021}, \cite{HO2021} apply to the case of the heat equation with the potential $x^2$ since the Green kernel is explicitly  known in this case. The description of the reachable space for the Hermite heat equation has been initiated in \cite{LR2020} where the authors proved that the functions which are holomorphic on a certain disk are reachable for the equation \eqref{eq-Hermiteheat}. Actually, their result holds for a general case of semilinear heat equations with very smooth controls (Gevrey of order $2$). When we restrict our attention on the specific equation \eqref{eq-Hermiteheat}, we can give a much more precise description. 
The first result is the same right inclusion as in \cite{HO2021}.
\begin{thm}
\label{thm-reach-segment-rightinclusion}
We have $\Ran(\Phi^{\mathrm{H}}_\tau) \subset A^2(D)$.
\end{thm}
While we were not able to prove equality, we have the following almost optimal result. For $\varepsilon >0$, denote by $D^{(\varepsilon)}$ the open square whose one diagonal is $(0, \, \pi+\varepsilon)$:
$$
D^{(\varepsilon)}=\{z=x+iy\in \C: |x-\frac{\pi+\varepsilon}{2}| + |y|<\frac{\pi+\varepsilon}{2}\}.
$$
\begin{thm}
\label{thm-reachable-segment}
For every $\varepsilon >0$, we have $A^2(D^{(\varepsilon)}) \subset \Ran(\Phi^{\mathrm{H}}_\tau)$. 
\end{thm}
Notice that the membership of the reachable states to the Bergman space has no interest for the right-hand sides of the square since $\varepsilon$ is arbitrary, but is sharp for the left-hand sides. \\

Finally, since the heat equation \eqref{eq-Heat} has constant coefficients, it is invariant by dilation and translation. Hence the reachable space of the heat equation on a general segment can be trivially deduced from the one of \eqref{eq-Heat}. The situation is completely different for the Hermite heat equation since the potential $x^2$ is space-dependent. For this reason we will also consider this equation on the segment $[-L, L]$ for $L>0$. 
\begin{equation}
\label{eq-Hermiteheat2}
\tag{HHE2}
\begin{cases}
\partial_t w  -\partial_x^2w + x^2 w = 0 , \quad t\in (0, \tau),\  x \in (-L, \, L) \\
w(t, -L)=u_{-L}(t), \qquad w(t, L)=u_L(t), \quad t\in (0, \tau) \\
w(0, x)=f, \quad x \in (-L, \, L).
\end{cases}
\end{equation}
Notice that this equation can be viewed as an equation on $[0, \pi]$ with a change of potential. Denoting by $D_L$ the square whose one diagonal is $(-L, L)$:
\begin{equation}\label{DL}
 D_L=\{z=x+iy\in \C: |x| + |y| < L\},
\end{equation}
and $\Phi^{\mathrm{H}}_{\tau, (-L, L)}$ the corresponding input-to-state map, we obtain the following result.
\begin{prop}
\label{prop-reach-segment-centered} 
Let $L>0$. For every $L'>L$, we have 
$$\Hol(D_{L'}) \subset \Ran\Phi^{\mathrm{H}}_{\tau, (-L, L)} \subset A^2(D_L).$$\end{prop} 
This result is not as sharp as the one obtained in Theorem \ref{thm-reachable-segment}, but it will be one ingredient of it. The left inclusion is the same as that given by Dardé and Ervedoza in \cite[Theorem 1.1]{DE18} for the classical 1-D heat equation.
Our proof is an adaptation to the Hermite heat equation of a very recent and clever argument of Strohmaier and Waters developed in \cite[Theorem 2]{SW2020} which generalizes the Dardé-Ervedoza result to the $n$-dimensional setting.

\subsection{On the half-line}
Here, we consider also the two previous equations but on the half-line. 
\begin{align}
&\begin{cases}
\label{eq-Heat-halfline}
\tag{HE-hl}
\partial_t w  -\partial_x^2 w = 0 , \quad t>0,\  x \in (0, \, \infty) \\
w(t, 0)=u_0(t) , \quad t>0 \\
w(0, x)=f, \quad x\in (0, \infty).
\end{cases}\\
&\begin{cases}
\label{eq-Hermiteheat-halfline}
\tag{HHE-hl}
\partial_t w  -\partial_x^2 w  +x^2 w = 0 , \quad t>0,\  x \in (0, \, \infty) \\
w(t, 0)=u_0(t) , \quad t>0 \\
w(0, x)=f, \quad x\in (0, \infty).
\end{cases}
\end{align}

As equations \eqref{eq-Heat} and \eqref{eq-Hermiteheat}, the two previous equations define well-posed linear time-invariant control systems: for every $f \in X=W^{-1,2}(0, \infty)$ and every $u_0 \in L^2\left((0, \infty); \, \C \right)$, each of these equations admits a unique solution $w \in C\left([0, \infty); W^{-1,2}(0, \infty)\right)$. 
We denote by $(\T_{t, 0})_{t \geq 0}$ (resp. $\left(\T_{t, 0}^{\mathrm{H}}\right)_{t \geq 0}$) the Dirichlet semigroup generated respectively by the (one-dimensional) Laplacian (resp. the harmonic oscillator); and by $\Phi_{\tau, 0}$ (resp. $\Phi^{\mathrm{H}}_{\tau, 0}$) the corresponding input-to-state map. 

However, unlike equations \eqref{eq-Heat} and \eqref{eq-Hermiteheat} the latter equations \eqref{eq-Heat-halfline} and \eqref{eq-Hermiteheat-halfline} are not null-controllable. 
Worse, as proved in \cite[Theorem 1.2 and Corollary 1.4] {DE19}, it turns out that the lack of null-controllability is maximal: there is no non trivial initial condition $f \in X$ which can be steered to zero, and this in whatever time. In other words, for any $0\neq f \in X$ and any $\tau >0$, we have $\T_{t, 0} f \notin \Ran \Phi_{\tau, 0}$ (resp. $\T_{t, 0}^{\mathrm{H}} f \notin \Ran \Phi^{\mathrm{H}}_{\tau, 0}$). 
This means that the reachable space of equation \eqref{eq-Heat-halfline} (resp. \eqref{eq-Hermiteheat-halfline}), given by $\T_{t, 0} f + \Ran \Phi_{\tau, 0}$ (resp.  $\T_{t, 0}^{\mathrm{H}} f + \Ran\Phi^{\mathrm{H}}_{\tau, 0}$) is an affine space which is linear only for $f=0$. In this last case, this space is called \emph{null-reachable space}. 
Notice that now it also depends on time.\\

The characterization of the null-reachable space of \eqref{eq-Heat-halfline} has been discussed in several papers. In \cite{AHS1990}, the authors proved using reproducing kernels arguments that 
$$ \Phi_{\tau, 0}\left[L^2((0, \tau); dt/t)\right]=A^2(\Delta, \, \omega_{0, \tau})$$
where $\Delta$ denotes the right-angle sector 
$$\Delta=\enstq{z \in \C}{|\arg(z)|< \frac{\pi}{4}}$$ 
and $\omega_{0, \tau}$ denotes the exponential weight
$$\omega_{0, \tau}(z)= \frac{1}{\tau} e^{\frac{\Re(z^2)}{2\tau}}.$$
Conversely, the second author proved in \cite{Ors2021} that 
\begin{equation}
\label{eq-inclusion-heat-halfline}
A^2(\Delta)= \Ran(\Phi_{\tau, 0}) \oplus X
\end{equation}
where $X$ is a suitable space of entire functions.
Thus, combining the above two results we have 
$$A^2(\Delta, \, \omega_{0, \tau}) \subset \Ran(\Phi_{\tau, 0}) \subset A^2(\Delta).$$
Actually, an exact characterization of the null-reachable $\Ran(\Phi_{\tau, 0})$ has been given in \cite{Saitoh91}. In order to state it, let us define another space of analytic functions. Let $\Omega$ be a simply connected domain in the complex plane with at least two boundary points. For a non negative measurable function $\omega$, we say that an holomorphic function $f$ belongs to the Hardy-Smirnov space (also called Szeg\"o space) $E^2(\Omega, \omega)$ if there exists a sequence of rectifiable Jordan curves $(\gamma_n)_{n \in \N}$ eventually surrounding each compact subdomain of $\Omega$ such that 
$$\sup_n \int_{\gamma_n} |f(z)|^2 \omega(z) |dz| < \infty.$$
Therefore the result can be formulated as: 
$$\Ran(\Phi_{\tau, 0}) = A^2(\Delta, \omega_{0, \tau}) + E^2\left(\Delta, 1/|z|\right).$$

In the present paper, we will determine the null-reachable space of equation \eqref{eq-Hermiteheat-halfline}. Since the potential vanishes at $x=0$ which is our point of control, it seems reasonable to expect that the reachable space is not too far from the reachable space of \eqref{eq-Heat-halfline}. Actually, we will prove that they are the same, up to a change of time. 
\begin{thm}
\label{thm-reach-half-line} 
Let $\tau >0$ and $T=\tanh(2\tau)/2$. Then we have 
$$ \Ran \Phi^{\mathrm{H}}_{\tau, \, 0} = \Ran \Phi_{T, 0}. $$
\end{thm}
This result, interesting in itself, and will be one key of the proof of Theorem \ref{thm-reach-segment-rightinclusion}. \\


The paper is organized as follows. In Section \ref{section-solution}, for the sake of completeness, we compute the solution of \eqref{eq-Hermiteheat} and \eqref{eq-Hermiteheat-halfline} using the method of images. Section \ref{S3} is devoted to the proof of Theorem \ref{thm-reach-half-line},
and Section \ref{S4} to that of Theorem \ref{thm-reach-segment-rightinclusion}. 
In Section \ref{S5} we prove Proposition \ref{prop-reach-segment-centered} and
in Section \ref{S6} Theorem
\ref{thm-reachable-segment}.



\bigskip



\section{Solution of the Hermite heat equation \label{section-solution}}
In this section, we compute the solution of equation \eqref{eq-Hermiteheat}. The method proposed in \cite{HKT2020} in order to tackle the reachable space problem for the heat equation \eqref{eq-Heat} was to compute the solution \emph{via} a decomposition on the orthonormal basis of eigenfunctions of the Dirichlet Laplacian (the sine Fourier basis) and then to use the Poisson summation formula. This allowed the authors to make the heat kernel of the whole real line appear. Actually, the final formula can be obtained alternatively using the method of images. This permits to avoid the Poisson formula step and so to generalize the final formula to other equations. It is this method we choose here in order to establish our formula for the solution of \eqref{eq-Hermiteheat} and \eqref{eq-Hermiteheat-halfline}, involving the Mehler kernel (which is the fundamental solution of the Hermite heat equation). 

\paragraph{Mehler kernel.}
Let us recall some facts about the fundamental solution of the Hermite heat equation on $\R$, i.e.
the solution of  
\begin{equation}
\label{Melher}
\begin{cases}
\partial_t K  -\partial_x^2 K  +x^2 K = 0 , \quad t>0,\  x \in \R \\
K(0, x, \cdot)=\delta_x.
\end{cases}
\end{equation}
where $\delta_x$ stands for the Dirac delta distribution at $x$. 
It is given by the Mehler kernel 
\begin{equation}
\label{eq-Mehler-kernel}
K(t,x,y)= \frac{1}{\sqrt{2\pi \sinh(2t)}} \exp\left(-\mathrm{coth}(2t) \frac{x^2 +y^2}{2} + \frac{xy}{\sinh(2t)}\right).
\end{equation}
It can be obtained using the normalized Hermite functions defined by 
$$h_k(x) = \frac{(-1)^k e^{\frac{x^2}{2}}}{2^k k! \pi^{\frac{1}{2}}}\frac{d^k}{dx^k} e^{-x^2}, \quad x \in \R, k \in \N$$
which are known to form an orthonormal basis of $L^2(\R)$ 
and to be the eigenfunctions of the operator $-\Delta + x^2$ on $L^2(\R)$ associated with the eigenvalues $\lambda_k=2k+1$. Thus, the solution of \eqref{Melher} can be derived decomposing on this orthonormal basis. Finally, the above form of the Mehler kernel is obtained using the so called Mehler formula. 

More general Mehler kernels can be found in \cite[Subsection 2.1]{DE19}.
\\


\paragraph{Computation of the solution using the method of images.}
Let $X:=W^{-1,2}(0, \pi)$. We recall that $ w \in C\left([0, \, \infty), \, X\right)$ is called solution of \eqref{eq-Hermiteheat} if it satisfies 
\begin{align}
\label{weak-sol-HHE} 
\prodscal{w(t)}{\psi}_{-1,\,1} =  & \prodscal{f}{\psi}_{-1,\,1}  +  \int_0^t \prodscal{w(s)}{\frac{d^2 \psi}{dx^2}-x^2 \psi}_{-1,\,1}ds  \notag \\ 
&+ \int_0^t u_0(s) \overline{\frac{d \psi}{dx}(0)}ds - \int_0^t u_\pi(s) \overline{\frac{d \psi}{dx}(\pi)} ds
\end{align}
for every $t \geq 0$ and every $\psi \in W^{2,2}(0, \, \pi) \cap W^{1,2}_0(0, \, \pi)$ such that $\frac{d^2 \psi}{dx^2} \in W^{1,2}_0(0, \, \pi)$ (see for instance \cite[Section 4]{TW2009}). In \eqref{weak-sol-HHE}, $\prodscal{\cdot}{\cdot}_{-1,1}$ denotes the duality  $W^{-1,2}(0, \, \pi) \, -\, W^{1,2}_0(0, \, \pi)$.\\

For ${\varphi \in \mathcal{D}(\R):= C^\infty_\mathrm{c}(\R)}$, we denote by $\varphi_{\mathrm{per}}$ the function 
$$\varphi_{\mathrm{per}}=\sum_{k \in \Z} \varphi(\cdot+2k\pi) - \varphi(2k\pi-\cdot),$$
which is odd and $2\pi$-periodic.
Remark that the sum is finite and $\varphi_{\mathrm{per}}, \frac{d^2 \varphi_{\mathrm{per}}}{dx^2} \in  W^{1,2}_0(0, \, \pi)$.  
Hence, for $T \in W^{-1, \, 2}(0, \, \pi)$, we can define $T_{\mathrm{per}} \in \mathcal{D}'(\R)$ by duality
$$ \forall \varphi \in \mathcal{D}(\R), \ \prodscal{T_{\mathrm{per}}}{\varphi}_{\mathcal{D}'(\R), \, \mathcal{D}(\R)} :=  \prodscal{T}{\overline{\varphi_{\mathrm{per}}}}_{-1, \, 1}.$$
It is easy to check that $T_{\mathrm{per}}$ is well defined as a distribution and $T_{\mathrm{per}} \in \mathcal{S}'(\R)$. Note that $T_{\mathrm{per}}$ is the continuation (in the sense of distributions) of $T$ to $\R$, first by odd extension to $(-\pi, \pi)$ and then by periodic extension to $\R$ with period $2\pi$. 

Given $L^2$ functions $u_0$ and $u_{\pi}$, let $w$ be the unique solution of \eqref{eq-Hermiteheat} with $f=0$ in the sense of \eqref{weak-sol-HHE}. Recall that $\mathcal{D}(\R_t) \otimes \mathcal{D}(\R_x)$ is dense in $\mathcal{D}(\R_t \times \R_x)$. Write $W$ the tempered distribution defined on $\R_t \times \R_x$ by 
$$\forall \zeta \in \mathcal{D}(\R_t), \ \forall \varphi \in \mathcal{D}(\R_x), \  \prodscal{W}{\zeta \otimes \varphi}_{\mathcal{D}'(\R_t \times \R_x), \, \mathcal{D}(\R_t \times \R_x)}:=\int_{\R_+} \prodscal{w(t)_{\mathrm{per}}}{\varphi}_{\mathcal{D}'(\R), \, \mathcal{D}(\R)} \zeta(t) dt. $$ 
Note that 
\begin{equation}
\label{eq-equalitydistrib-sol}
W_{|\R_+^* \times (0, \, \pi)}= w \quad \text{(in the sense of distributions)}.
\end{equation}

\begin{lem}
\label{lemma-equationdistribHHE}
The distribution $W$ is the unique solution in $\mathcal{S}'(\R_t \times \R_x)$ of the equation
\begin{equation}
\label{eq-Hermiteheat-on-R}
\begin{cases}
\partial_t W  -\partial_x^2W  +x^2 W = -2u_0 \mathds{1}_{\R_+}  \otimes \sum_{k \in \Z} \delta'_{2k\pi} \ + \ 2 u_\pi\mathds{1}_{\R_+}  \otimes  \sum_{k \in \Z} \delta'_{(2k+1)\pi}=:F \\
\mathrm{supp}(W) \subset \R_+ \times \R.
\end{cases}
\end{equation}
\end{lem}
\begin{proof}
Since $W$ and $F$ belong to $\mathcal{S}'(\R_t \times \R_x)$ (recall that $L^2 \subset \mathcal{S}'$), it suffices to prove the equality in $\mathcal{D}'(\R_t \times \R_x)$. 
For every $\zeta \in \mathcal{D}(\R_t)$ and every $\varphi \in \mathcal{D}(\R_x)$, 
\begin{align*}
&\prodscal{\partial_t W}{\ \zeta \otimes \varphi}_{\mathcal{D}'(\R_t \times \R_x), \, \mathcal{D}(\R_t \times \R_x)}= -\prodscal{W}{\ \partial_t \zeta \otimes \varphi}_{\mathcal{D}'(\R_t \times \R_x), \, \mathcal{D}(\R_t \times \R_x)}\\
&=-\int_{\R_+} \prodscal{w(t)_{\mathrm{per}}}{\varphi}_{\mathcal{D}'(\R), \, \mathcal{D}(\R)} \frac{d \zeta}{dt}(t) dt\\
&=\int_{\R_+} \frac{d}{dt}\left(\prodscal{w(t)_{\mathrm{per}}}{\varphi}_{\mathcal{D}'(\R), \, \mathcal{D}(\R)}\right)  \zeta(t) dt.
\end{align*}

But, for each $t \geq 0$ and for every $\varphi \in \mathcal{D}(\R)$, we have (using \eqref{weak-sol-HHE})
\begin{align*}
&\prodscal{w(t)_{\mathrm{per}}}{\varphi}_{\mathcal{D}'(\R), \, \mathcal{D}(\R)} = \prodscal{w(t)}{\overline{\varphi_{\mathrm{per}}}}_{-1, \, 1} \\
&= \int_0^t \prodscal{w(s)}{\ \frac{d^2 \overline{\varphi_{\mathrm{per}}}}{dx^2}-x^2 \overline{\varphi_{\mathrm{per}}}}_{-1,\,1}ds
+ \int_0^t u_0(s) \frac{d \varphi_{\mathrm{per}}}{dx}(0)ds - \int_0^t u_\pi(s) \frac{d\varphi_{\mathrm{per}}}{dx}(\pi) ds\\
&= \int_0^t \prodscal{w(s)}{\ \overline{\left(\frac{d^2 \varphi}{dx^2}-x^2 \varphi \right)_{\mathrm{per}}}}_{-1,\,1}ds\\
&\qquad + \int_0^t 2 u_0(s) \sum_{k \in \Z} \frac{d \varphi}{dx}(2k\pi) ds - \int_0^t 2 u_\pi(s) \sum_{k \in \Z} \frac{d \varphi}{dx}((2k+1)\pi) ds \\
&= \int_0^t \prodscal{w(s)_{\mathrm{per}}}{\ \frac{d^2 \varphi}{dx^2}-x^2 \varphi }_{\mathcal{D}', \, \mathcal{D}}ds\\
&\quad -2 \int_0^t u_0(s) \prodscal{\sum_{k \in \Z} \delta'_{2k\pi}}{\varphi}_{\mathcal{D}', \, \mathcal{D}} ds + 2 \int_0^t u_\pi(s) \prodscal{\sum_{k \in \Z} \delta'_{(2k+1)\pi}}{\varphi}_{\mathcal{D}', \, \mathcal{D}} ds 
\end{align*}

Therefore, we get 
\begin{align*}
&\prodscal{\partial_t W}{\ \zeta \otimes \varphi}_{\mathcal{D}'(\R_t \times \R_x), \, \mathcal{D}(\R_t \times \R_x)}\\
&=\int_{\R_+} \left[ \prodscal{w(t)_{\mathrm{per}}}{\ \frac{d^2 \varphi}{dx^2}-x^2 \varphi }_{\mathcal{D}', \, \mathcal{D}} -2 u_0(t) \prodscal{\sum_{k \in \Z} \delta'_{2k\pi}}{\varphi}_{\mathcal{D}', \, \mathcal{D}}\right. \\
& \qquad \qquad \qquad \left. + 2 u_\pi(t) \prodscal{\sum_{k \in \Z} \delta'_{(2k+1)\pi}}{\varphi}_{\mathcal{D}', \, \mathcal{D}} \right] \zeta(t) dt \\
&= \prodscal{(\partial^2_x - x^2) W}{\ \zeta \otimes \varphi}_{\mathcal{D}', \, \mathcal{D}} - 2 \prodscal{u_0 \mathds{1}_{\R_+}  \otimes \sum_{k \in \Z} \delta'_{2k\pi}}{\ \zeta \otimes \varphi}_{\mathcal{D}', \, \mathcal{D}} \\
& \qquad \qquad \qquad + 2 \prodscal{u_\pi\mathds{1}_{\R_+}  \otimes  \sum_{k \in \Z} \delta'_{(2k+1)\pi}}{\ \zeta \otimes \varphi}_{\mathcal{D}', \, \mathcal{D}}.
\end{align*}

\end{proof}

\begin{prop}
\label{prop-solHHE}
The solution of the Hermite heat equation \eqref{eq-Hermiteheat} with $f=0$ is given by
\begin{equation}
\label{eq-solution-HHE-formula}
\begin{split}
\forall t >0, \ \forall x \in (0, \pi), \quad w(t,x)= 2 &\int_0^t \sum_{k \in \Z} \frac{\partial K}{\partial y}(t-s, x, 2k\pi) u_0(s) ds \\
&- 2 \int_0^t \sum_{k \in \Z} \frac{\partial K}{\partial y}(t-s, x, (2k+1)\pi) u_\pi(s) ds.
\end{split}
\end{equation}
\end{prop}
\begin{proof}
By Lemma \ref{lemma-equationdistribHHE}, $W$ is the unique solution of equation \eqref{eq-Hermiteheat-on-R} in $\mathcal{S}'$
which is by Duhamel's formula
\begin{align*}
W(t, x)&= \prodscal{F}{\ K(t-\cdot \, , \,  x, \, \cdot \,)}_{\mathcal{S}'(\R_s \times \R_y), \, \mathcal{S}(\R_s \times \R_y)} \\
&= -2\int_0^t \prodscal{\sum_{k \in \Z} \delta'_{2k\pi}}{\ K(t-s, x, \cdot)}_{\mathcal{S}', \, \mathcal{S}} u_0(s) ds \\
& \qquad + 2 \int_0^t \prodscal{\sum_{k \in \Z} \delta'_{(2k+1)\pi}}{\ K(t-s, x, \cdot)}_{\mathcal{S}', \, \mathcal{S}}  u_\pi(s) ds \\
&= 2 \int_0^t \sum_{k \in \Z} \frac{\partial K}{\partial y}(t-s, x, 2k\pi) u_0(s) ds \\
&\qquad - 2 \int_0^t \sum_{k \in \Z} \frac{\partial K}{\partial y}(t-s, x, (2k+1)\pi) u_\pi(s) ds
\end{align*}
The result follows from \eqref{eq-equalitydistrib-sol}.
\end{proof}

With an analogous method, we get the solution of \eqref{eq-Hermiteheat-halfline} with $f=0$ which is 
\begin{equation}
\label{eq-solHHEhl-formula}
\Phi^{\mathrm{H}}_{t, \, 0} u_0 (x)= 2 \int_0^t \frac{\partial K}{\partial y}(t-s, x, 0) u_0(s) ds.
\end{equation}
It can be noticed that it corresponds to the zero term of the first sum in \eqref{eq-solution-HHE-formula}. In the same way, we denote by $\Phi^{\mathrm{H}}_{t, \, \pi} u_\pi$ the zero term of the second sum in \eqref{eq-Hermiteheat-halfline}, which is actually the solution of the Hermite heat equation on $(-\infty, \pi)$ with Dirichlet boundary condition $u_\pi$ at $x=\pi$ and with null initial condition. This leads to the following writing 
\begin{equation}
\label{eq-solHHE-formula2} 
\Phi^{\mathrm{H}}_t (u_0, u_\pi)=\Phi^{\mathrm{H}}_{t, \, 0} u_0 + \Phi^{\mathrm{H}}_{t, \, \pi} u_\pi  + R_{t, \, 0} u_0   + R_{t, \, \pi} u_\pi 
\end{equation}
where $R_{t, \, 0} u_0$   and $R_{t, \, \pi} u_\pi$ are the remainder terms
\begin{equation}
\label{eq-rest1}R_{t, \, 0} u_0 (x) = 2 \int_0^t \sum_{k \in \Z^*} \frac{\partial K}{\partial y}(t-s, x, 2k\pi) u_0(s) ds,
\end{equation}
\begin{equation}
\label{eq-rest2}
R_{t, \, \pi} u_\pi (x)= - 2 \int_0^t \sum_{k \in \Z^*} \frac{\partial K}{\partial y}(t-s, x, (2k+1)\pi) u_\pi(s) ds.
\end{equation}

\section{Null-reachable space on the half-line \label{section-reachable-HHE-half-line}}\label{S3}
In this section, we prove Theorem \ref{thm-reach-half-line}. 
 \begin{proof}[Proof of Theorem \ref{thm-reach-half-line}]
 First, let us recall the expression of $\Phi_{T, 0}$ (see e.g. \cite{HKT2020} or \cite{Ors2021}).
 $$\Phi_{T, 0}f (z) =  \int_0^T \frac{z e^{- \frac{z^2}{4(T-\sigma)}}}{2\sqrt{\pi} {(T-\sigma)}^{\frac{3}{2}}} f(\sigma) d\sigma.$$
Then, using the change of variable $x=\alpha(s):=\tanh(2(\tau-s))/2$ (which leads to $ds=-\frac{dx}{1-4x^2}$ and $\sinh(2(\tau-s))=\frac{2x}{\sqrt{1-4x^2}}$), we obtain 
\begin{align*}
\Phi^{\mathrm{H}}_{\tau, \, 0}u_0(z)&:=2\int_{0}^\tau \frac{\partial K}{\partial y}(\tau - s, z , 0) u_0(s)\, ds \\
&= \frac{2z}{\sqrt{2 \pi}}\int_0^\tau \frac{e^{-\coth(2(\tau-s))z^2/2}}{\sinh(2(\tau-s))^{3/2}} u_0(s) ds\\
&= \frac{z}{2 \sqrt{\pi}} \int_0^{\frac{\tanh(2\tau)}{2}} \frac{e^{-\frac{z^2}{4x}}}{x^{3/2}} (1-4x^2)^{-1/4} u_0(\alpha^{-1}(x)) dx.
\end{align*}
Now, setting $T=\tanh(2\tau)/2$ and the change of variable $x=T-t$ leads to 
$$\left[\Phi^{\mathrm{H}}_{\tau, \, 0}u_0\right](z)= \left[\Phi_{T, 0} \widetilde{u_0}\right](z) $$
where $\widetilde{u_0}$ is given by 
$$\widetilde{u_0}(t)=(1-4(T-t)^2)^{-1/4} u_0(\alpha^{-1}(T-t)).$$
Since $\tanh(2\tau)/2<1/2$, it is clear that the operator $S: L^2((0, \tau); \C) \to L^2((0, T); \C)$ defined by 
$$Su_0=\widetilde{u_0}$$
is boundedly invertible and we have 
$$\int_0^T |\widetilde{u_0}(t)|^2 dt = \int_0^\tau \frac{1}{\cosh(2(\tau-s))} |u_0(s)|^2 ds \asymp \int_0^\tau |u_0(s)|^2ds.$$ 
Hence, 
$$\Ran \Phi^{\mathrm{H}}_{\tau, \, 0} = \Ran \Phi_{T, 0}$$
and the proof is complete. 
\end{proof}

An immediate consequence of the above discussion is that $\Phi^H_{\tau,0}$ is bounded
from $L^2(0,\tau)$ to $A^2(D)$ which we state here for later purposes as a separated result (and even to a suitably weighted Bergman space on the sector $\Delta$).
\begin{cor}\label{Cor7}
The operator $\Phi^H_{0,\tau}$ is bounded from $L^2(0,\tau)$ to $A^2(D)$.
\end{cor}

\section{Proof of Theorem \ref{thm-reach-segment-rightinclusion}}
\label{S4}
\begin{proof}[Proof of Theorem \ref{thm-reach-segment-rightinclusion}]
Let $u_0, u_\pi \in L^2([0, \tau]; \C)$. 
In view of equalities \eqref{eq-solHHEhl-formula}, \eqref{eq-solHHE-formula2}, \eqref{eq-rest1} and \eqref{eq-rest2}, we will prove that each term belongs to $A^2(D)$. 
With the explicit structure of the Mehler kernel in mind (see \eqref{eq-Mehler-kernel}), it is easy to show that $R_{\tau, \, 0} u_0, \, R_{\tau, \, \pi} u_\pi \in A^2(D)$.
Indeed, consider for instance $R_{\tau,0}$. Similarly as in \cite[Proposition 1.2]{HKT2020},
we get
\begin{eqnarray*}
|R_{\tau,0}u_0(z)|
 &\le&2\sum_{k\in\Z^*}\int_0^{\tau}|\frac{\partial K}{\partial y}(\tau-s),z,2k\pi)u_0(s)|ds\\
 &\le& 2\sum_{k\in\Z^*}\left(\int_0^{\tau}|\frac{\partial K}{\partial y}(\tau-s),z,2k\pi)|^2ds\right)^{1/2}
 \|u_0\|_2,
\end{eqnarray*}
and, since $\Re (z^2)>0$,
\begin{eqnarray*}
|\frac{\partial K}{\partial y}(t,z,2k\pi)|
&\lesssim& \frac{1+2k\pi}{t^{3/2}}
 \exp\left(-\coth(2t)\frac{\Re z^2+(2k\pi)^2}{2}+\frac{\Re z \times 2k\pi}{\sinh(2t)}\right)\\
&\le&  \frac{1+2k\pi}{t^{3/2}}
 \exp\left(-\coth(2t)\frac{(2k\pi)^2}{2}+\frac{ 2|k|\pi^2}{\sinh(2t)}\right)\\
&\le& \frac{1+2k\pi}{t^{3/2}}
 \exp\left(\frac{-2k^2\pi^2+2|k|\pi^2}{\sinh(2t)}\right),\\
\end{eqnarray*}
which yields the required integrability condition in $t$ and the summability for 
$k\in \Z^*\setminus\{1\}$. When $k=1$, the above estimate is too crude, and one has
to distinguish between $z$ being in $D$ close to $\pi$ (say $\Re z>3\pi/4$), 
in which case $\Re z^2/2$ in the
first term in the exponential adds the necessary convergence, and $z$ being far from
$\pi$ (say $\Re z\le 3\pi/4$), in which case the second term in the exponential
is controlled by $3\pi^2/2$ which yields the required convergence.
The term $R_{\tau,\pi}$ is treated in the same fashion.

Next, 
by Corollary \ref{Cor7} we have $\Phi^{\mathrm{H}}_{\tau, \, 0} u_0 \in A^2(D)$. Hence, it remains to prove that the range of $\Phi^{\mathrm{H}}_{\tau, \, \pi}$ is contained in $A^2(D)$. Since this operator is the input-to-state map of the Hermite heat equation on $(-\infty, \pi)$ as explained after \eqref{eq-solHHEhl-formula}, we upgrade this result as a lemma which concludes the proof of the theorem. 

\begin{lem}
\label{lemma8}
We have $\Ran \Phi^{\mathrm{H}}_{\tau, \, \pi} \subset A^2(D)$. 
\end{lem}
\begin{proof}
Let us write the expression of  $\Phi^{\mathrm{H}}_{\tau, \, \pi}u_\pi$:
\begin{align*}
&\left[\Phi^{\mathrm{H}}_{\tau, \, \pi}u_\pi\right](z):=-2\int_{0}^\tau \frac{\partial K}{\partial y}(\tau - s, z , \pi) u_\pi(s) \, ds \\
&\quad=-2\int_0^\tau \left(-\coth(2(\tau-s)) \pi + \frac{z}{\sinh(2(\tau-s))}\right) \frac{e^{-\coth(2(\tau-s))\frac{z^2+\pi^2}{2} + \frac{\pi z}{\sinh(2(\tau-s))}}}{\sqrt{2 \pi\sinh(2(\tau-s))}} u_{\pi}(s)ds.\\
&\quad=-2\int_0^\tau \left(-\coth(2(\tau-s)) \pi + \frac{z}{\sinh(2(\tau-s))}\right) \frac{e^{-\frac{\coth(2(\tau-s))}{2}(\pi-z)^2 + \frac{1-\cosh(2(\tau-s))}{\sinh(2(\tau-s))}\pi z}}{\sqrt{2 \pi\sinh(2(\tau-s))}} u_{\pi}(s)ds.
\end{align*}

Note that the equality 
$(\Phi_{\tau, \, \pi}f) (z) = -(\Phi_{\tau, \, 0}f)(\pi-z)$
which holds for the classical heat equation is no longer true here. However, we can make  $\Phi^{\mathrm{H}}_{\tau, \, 0}$ appear in the above expression of $\Phi^{\mathrm{H}}_{\tau, \, \pi} u_\pi$.
For that, remark that 
\begin{equation}\label{Fct-psi}
\psi(s):= \frac{1-\cosh(2(\tau-s))}{\sinh(2(\tau-s))}\sim -(\tau-s),
\end{equation} 
and $\psi$ is bounded on $[0, \tau]$. 
Hence writing $e^{\psi(s)\pi z}= \sum_{n=0}^\infty \frac{(\psi(s)\pi z)^n}{n!}$, we obtain 
\begin{align*}
&\left[\Phi^{\mathrm{H}}_{\tau, \, \pi}u_\pi \right](z)\\
&=-2\int_0^\tau  \sum_{n=0}^\infty \frac{(\psi(s)\pi z)^n}{n!} \left(-\coth(2(\tau-s)) \pi + \frac{z}{\sinh(2(\tau-s))}\right) \frac{e^{-\frac{\coth(2(\tau-s))}{2}(\pi-z)^2}}{\sqrt{2 \pi\sinh(2(\tau-s))}} u_{\pi}(s)ds\\
&= -2 \sum_{n=0}^\infty \frac{(\pi z)^n}{n!}  \int_0^\tau (\psi(s))^n \left(-\coth(2(\tau-s)) \pi + \frac{z}{\sinh(2(\tau-s))}\right) \frac{e^{-\frac{\coth(2(\tau-s))}{2}(\pi-z)^2}}{\sqrt{2 \pi\sinh(2(\tau-s))}} u_{\pi}(s)ds.
\end{align*}
Now, since $-\coth(2(\tau-s) )\pi + \frac{z}{\sinh(2(\tau-s))} = \frac{z-\pi}{\sinh(2(\tau-s))} + \pi \psi(s)$, it follows 
\begin{align*}
&\left[\Phi^{\mathrm{H}}_{\tau, \, \pi}u_\pi \right](z)\\
&= \sum_{n=0}^\infty \frac{(\pi z)^n}{n!}  \frac{2}{\sqrt{2\pi}} \int_0^\tau (\psi(s))^n  \frac{\pi-z}{\sinh(2(\tau-s))^{\frac{3}{2}}} e^{-\frac{\coth(2(\tau-s))}{2}(\pi-z)^2} u_{\pi}(s)ds \\
&\quad - \sqrt{2 \pi} \sum_{n=0}^\infty \frac{(\pi z)^n}{n!}  \int_0^\tau (\psi(s))^{n+1}  \frac{e^{-\frac{\coth(2(\tau-s))}{2}(\pi-z)^2}}{\sqrt{\sinh(2(\tau-s))}} u_{\pi}(s)ds \\
&= \sum_{n=0}^\infty \frac{(\pi z)^n}{n!} \left[\Phi^{\mathrm{H}}_{\tau, \, 0}(\psi^{n} u_\pi)\right](\pi-z) \\
&\quad - \sqrt{2 \pi} \sum_{n=0}^\infty \frac{(\pi z)^n}{n!}  \int_0^\tau (\psi(s))^{n+1}  \frac{e^{-\frac{\coth(2(\tau-s))}{2}(\pi-z)^2}}{\sqrt{\sinh(2(\tau-s))}} u_{\pi}(s)ds\\
&:=\sum_{n=0}^\infty \frac{(\pi z)^n}{n!} \left(A(z) -  B(z)\right). 
\end{align*}
We treat these two terms separately. For the first term we 
use Corollary \ref{Cor7}
\begin{align*}
\|A\|_{A^2(D)} &=   \|\left[\Phi^{\mathrm{H}}_{\tau, \, 0}(\psi^{n} u_\pi)\right](\pi-\cdot)\|_{A^2(D)}\\
&= \|\Phi^{\mathrm{H}}_{\tau, \, 0}(\psi^{n} u_\pi)\|_{A^2(D)} \\
&\lesssim  \|\psi^{n} u_\pi \|_{L^2(0, \tau)}\\
&\leq \|u_\pi \|_{L^2(0, \tau)} \|\psi\|_{L^\infty(0, \infty)}^{n} 
\end{align*}
On the other hand, we have an integral which converges easily. Indeed, using the Cauchy-Schwarz inequality
$$
\|B\|^2_{A^2(D)} \lesssim \|u_\pi \|^2_{L^2(0, \tau)} \int_0^\tau \int_D |\psi(s)|^{2n+2} \left|\frac{e^{-\coth(2(\tau-s))(\pi-z)^2}}{\sinh(2(\tau-s))} \right| dA(z) ds.$$
So, using that $|e^{-\coth(2(\tau-s))(\pi-z)^2}| \leq 1$ on $D$, we obtain 
\begin{align*}
 \|B\|^2_{A^2(D)} 
&\lesssim  \|u_\pi \|^2_{L^2(0, \tau)} \int_0^\tau  |\psi(s)|^{2n+2} \left|\frac{1}{\sinh(2(\tau-s))} \right| ds \\
&\leq  \|u_\pi \|^2_{L^2(0, \tau)} \|\psi\|_{L^\infty(0, \infty)}^{2n+1}  \int_0^\tau  \psi(s) \frac{1}{\sinh(2(\tau-s))}ds< \infty,
\end{align*}
which in view of \eqref{Fct-psi} is bounded.
Finally, since $|z|<\pi$, we have 
\begin{align*}
\|\Phi^{\mathrm{H}}_{\tau, \, \pi}u_\pi \|_{A^2(D)} &\leq \sum_{n=0}^\infty \frac{\pi^{2n}}{n!}  \left(\|A\|_{A^2(D)} + \|B\|_{A^2(D)} \right)\\
&\lesssim\sum_{n=0}^\infty \frac{\pi^{2n}}{n!} \left( \|u_\pi \|_{L^2(0, \tau)} \|\psi\|_{L^\infty(0, \infty)}^{n} + \|\psi\|_{L^\infty(0, \infty)}^{n} \|u_\pi \|_{L^2(0, \tau)} \right)\\
&=  \|u_\pi \|_{L^2(0, \tau)} \sum_{n=0}^\infty \frac{(\|\psi\|_{L^\infty(0, \infty) }\pi^2)^n}{n!}. 
\end{align*}
Since the last sum converges (to $e^{\|\psi\|_{L^\infty(0, \infty)} \pi^2}$), the proof is complete. 
\end{proof}
This ends also the proof of Theorem \ref{thm-reach-segment-rightinclusion}
\end{proof}

The converse inclusion is more tricky if one wants to use the same kind of arguments as in the case of the heat equation. While $\Phi^{\mathrm{H}}_{\tau, \, 0}$ is still isomorphic, the situation for $\Phi^{\mathrm{H}}_{\tau, \, \pi}$ is more complicated and awaits further investigation. 

\begin{rmk}
\label{rmk}
Since the potential $x^2$ is even, one could want to take a symmetric interval $(-L, L)$ with respect to zero. This will be done in Section \ref{section-centered-interval}. In that setting, the symmetry $\left[\Phi^{\mathrm{H}}_{\tau, L}f\right](z)=-\left[\Phi^{\mathrm{H}}_{\tau, L}f\right](-z)$ holds but the difficulty of describing the ranges of these operators is comparable to the difficulty of describing the range of $\Phi^{\mathrm{H}}_{\tau, \pi}$.
\end{rmk}

\section{Reachable space for a centered interval \label{section-centered-interval}}\label{S5}
In the present section, we prove Proposition \ref{prop-reach-segment-centered}. 
%
%
%
%
%
%
%
%
%
%
%
%
%
%
%

As explained in Remark \ref{rmk}, the right inclusion can be proved exactly as Theorem \ref{thm-reach-segment-rightinclusion}. Indeed, the discussions in Section \ref{section-solution} applied to the symmetric interval $(-\pi,\pi)$ with
$$
 \varphi_{per}(x)=\sum_{k\in\Z}\varphi(x+4k\pi)-\varphi((4k+2)\pi-x)
$$
lead, as in Proposition \ref{prop-solHHE}, to
\begin{align*}
w(t,x)= 2 &\int_0^t \sum_{k \in \Z} \frac{\partial K}{\partial y}(t-s, x, (4k-1)\pi) u_{-\pi}(s) ds \\
&- 2 \int_0^t \sum_{k \in \Z} \frac{\partial K}{\partial y}(t-s, x, (4k+1)\pi) u_\pi(s) ds.
\end{align*}
Now, the terms for $k=0$ in the sums are treated in the same way as $\Phi^{\mathrm{H}}_{\tau, \pi}$ in Lemma \ref{lemma8}, and the remainders terms are again very regular. 
\\

Let us thus focus on the left inclusion. 
It is based on the following key lemma whose proof is an adaptation of that for the heat equation in \cite[Theorem 3]{SW2020}. Recall that $D_L$ is the square whose one diagonal is the
interval $(-L,L)$ as defined in \eqref{DL}.
\begin{lem}
\label{lemma-SW}
Let $w_0 \in C^\infty_0(\R)$ admit an analytic extension on $D_L$ which we will also denote
by $w_0$.
If $w$ is a solution of the Hermite heat equation 
\begin{equation}
\notag
\begin{cases}
\partial_t w  -\partial_x^2 w  +x^2 w = 0 , \quad t >0,\  x \in \R \\
w(0, x)=w_0(x), \quad x\in \R
\end{cases}
\end{equation} 
then $w(t, \cdot)$ converges uniformly on compact subsets of $D_L$ to $w_0$ when $t \to 0$ .
\end{lem}
\begin{proof}

We consider a compact of $D_L$ that we can assume to be of the form $\overline{D_{L_1}}$ with $0<L_1<L$. Let $0<L_1<L_2<L$.  
We have 
\begin{align*}
w(t,x)= \int_\R K(t, x, \xi) w_0(\xi) d\xi&= \int_{-L_2}^{L_2} K(t, x, \xi) w_0(\xi) d\xi + \int_{\R \setminus (-L_2, L_2)} K(t, x, \xi) w_0(\xi) d\xi \\
&:=w_1(t,x) + w_2(t,x).
\end{align*}
Let us start with $w_2$. In this case we have $|\xi|^2>L_2^2$. Recall also that 
$$K(t,z,\xi)=  \frac{1}{\sqrt{2\pi \sinh(2t)}} e^{-\mathrm{coth}(2t) \frac{z^2 +\xi^2}{2} + \frac{z\xi}{\sinh(2t)}}.$$
Provided that $z \in D_{L_1}$, and hence $\Re (z^2)>-L_1^2$, we have $\Re(z^2 + \xi^2)>0$. Hence, using that $\cosh \geq 1$ 
we obtain 
\begin{equation}
\label{eq-inequality}
 \left|\exp\left(-\mathrm{coth}(2t) \frac{z^2 +\xi^2}{2} + \frac{z\xi}{\sinh(2t)}\right)\right| \leq \exp\left(-\frac{1}{2\sinh(2t)}\Re((z-\xi)^2)\right).
\end{equation}
Now, since for every $z\in D_{L_1}$ we have $|\Im (z)|\le L_1-|\Re (z)|$, and
for every $\xi \in \R \setminus (-L_2, L_2) $ 
we have $||\Re (z)| - \xi|\ge L_2-|\Re (z)|$, we get
\begin{align*}
\mathrm{Re}((z-\xi)^2) = (\mathrm{Re}(z)-\xi)^2 - \Im(z)^2 &\ge  (|\mathrm{Re}(z)| - L_2)^2 - (L_1-|\mathrm{Re}(z)|)^2\\
&\ge (L_2-L_1)(L_1 + L_2 - 2 |\mathrm{Re}(z)|) \\
& \ge (L_2 - L_1)^2 >0 
\end{align*}
Hence, $z \mapsto w_2(t, z)$ extends holomorphically to $D_{L_1}$.
Moreover, we have for every $z \in D_{L_1}$ and every $t>0$
$$|w_2(t, z)| =  \frac{1}{\sqrt{2\pi \sinh(2t)}}  e^{-\frac{1}{2\sinh(2t)} (L_2-L_1)^2} \|w_0\|_{L^1(\R)}.$$
Hence, $t \mapsto w_2(t,z)$ converges to $0$ uniformly on $\overline{D_{L_1}}$ as $t$ tends to $0$. 

Before treating $w_1$, we remind some facts on the Mehler kernel (see e.g. \cite[Section 3]{Dhungana2006}). It decomposes as 
\begin{equation}\label{K-heat}
K(t,x, \xi)= \widetilde{\eta}(t,x) \widetilde{K}(t, x, \xi),
\end{equation} 
where 
\begin{align*}
&\widetilde{\eta}(t,x) = \frac{1}{\sqrt{\cosh(2t)}} e^{-\frac{\tanh(2t)}{2}x^2}, \\
&\widetilde{K}(t, x, \xi) = \sqrt{\frac{\coth(2t)}{2\pi}} e^{-\frac{\coth(2t)}{2} \left(\xi- \frac{1}{\cosh(2t)} x\right)^2}.
\end{align*}
Notice that $\widetilde{\eta}$ does not depend on $\xi$ and $\widetilde{\eta}$ goes uniformly to 
zero on compact sets for $x$ (even in $\C$) when $t\to 0$. Observe that $\widetilde{K}$ is a 
heat kernel in $\xi$ so that an immediate computation
(see also \cite[Lemma 3.1]{Dhungana2006}) yields
for every $x \in \R$:
\begin{equation}
\label{eq-intK1tilde}
\int_\R \widetilde{K}(t, x, \xi) d\xi = 1
\end{equation}
and for every $\delta >0$
\begin{equation}
\label{eq-unityapprox}
\int_{\left|\xi- \frac{1}{\cosh(2t)} x\right| \ge \delta} \widetilde{K}(t, x, \xi) d\xi \to 0 \text{ as } t \to 0 \text{ uniformly for } x \in \R
\end{equation}

We turn back to the proof. For $w_1$, the expression of $\widetilde{K}$ leads to an application the Cauchy formula on the trapezoid delimited by the segment $[-L_2, L_2]$, the line of equation $y=\frac{\Im(z)}{\cosh(2t)}$ (notice that $\left|\frac{\Im(z)}{\cosh(2t)}\right| \le |\Im(z)|$) and the sides of $D_{L_2}$ (see Figure \ref{fig-trapezoid}).
\begin{figure}[h]
\begin{center}
\begin{tikzpicture}[scale=0.8]
\draw[->,color=black] (-4.786476139740749,0.) -- (5.826614872316539,0.);
\foreach \x in {-4.,-3.,-2.,-1.,1.,2.,3.,4.,5.}
\draw[shift={(\x,0)},color=black] (0pt,-2pt);
\draw[->,color=black] (0.,-4.08546540460618) -- (0.,4.496499568936687);
\draw [line width=1.pt,color=black] (4.,0.)-- (0.,4.);
\draw [line width=1.pt,color=black] (0.,4.)-- (-4.,0.);
\draw [line width=1.pt,color=black] (-4.,0.)-- (0.,-4.);
\draw [line width=1.pt,color=black] (0.,-4.)-- (4.,0.);
\draw [line width=1.pt,color=black] (3.,0.)-- (0.,3.);
\draw [line width=1.pt,color=black] (0.,3.)-- (-3.,0.);
\draw [line width=1.pt,color=black] (-3.,0.)-- (0.,-3.);
\draw [line width=1.pt,color=black] (0.,-3.)-- (3.,0.);
\draw [line width=2pt,color=red] (-4.,0.)-- (4.,0.)node[midway,sloped,xscale=1, yscale=1, line width=5pt]{$>$};
\draw [line width=2pt,color=red] (4.,0.)-- (2.48,1.52)node[midway,sloped,xscale=1, yscale=1]{$<$};
\draw [line width=2pt,color=red] (2.48,1.52)-- (-2.48,1.52)node[midway,sloped,xscale=1, yscale=1]{$<$};
\draw [line width=2pt,color=red] (-2.48,1.52)-- (-4.,0.)node[midway,sloped,xscale=1, yscale=1]{$<$};
\draw (-0.5,-0.1) node[anchor=north west] {\footnotesize{0}};
\draw (3.9967715943756272,0) node[anchor=north west] {\footnotesize{$L_2$}};
\draw (2.8,0) node[anchor=north west] {\footnotesize{$L_1$}};
\draw (-3.6336748746379746,0) node[anchor=north west] {\footnotesize{$-L_1$}};
\draw (-4.9,0) node[anchor=north west] {\footnotesize{$-L_2$}};
\draw [line width=1.pt,dash pattern=on 4pt off 4pt,domain=-4.786476139740749:5.826614872316539] plot(\x,{(--7.5296-0.01*\x)/4.97});
\draw (6,2) node[anchor=north west] {\footnotesize{$y=\frac{\Im(z)}{\cosh(2t)}$}};
\draw (0.20899600903793997,2.282389202628186) node[anchor=north west] {\footnotesize{$z$}};
\draw [line width=1.pt,dash pattern=on 4pt off 4pt] (0.66,1.8)-- (0.66,-0.01);
\draw (0.3004881729349856,0.02) node[anchor=north west] {\footnotesize{$\mathrm{Re}(z)$}};
\begin{scriptsize}
\draw [fill=blue] (0.66,1.8) circle (2pt);
\end{scriptsize}
\end{tikzpicture}
\end{center}
\caption{\label{fig-trapezoid} In red the integration path.}
\end{figure}
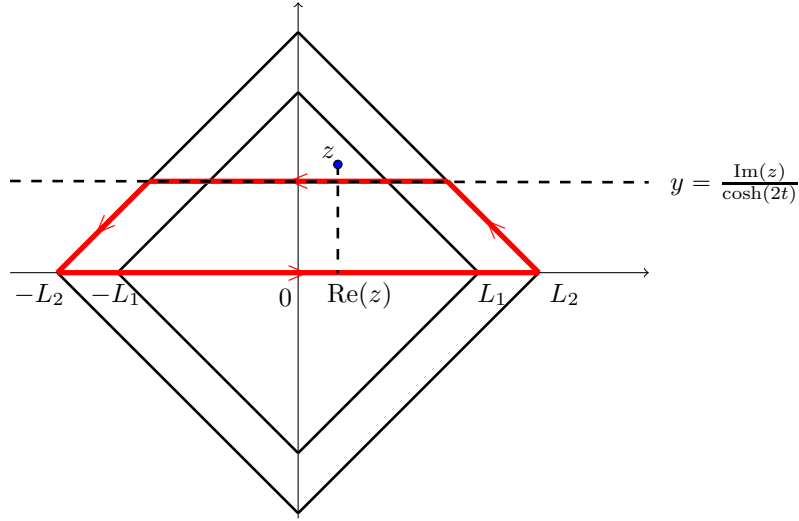

We obtain 
$$w_1(t,z)= \int_{-L_2 + \frac{\Im(z)}{\cosh(2t)}}^{L_2 - \frac{\Im(z)}{\cosh(2t)}}K\left(t,z, \xi + i \frac{\Im(z)}{\cosh(2t)}\right) \, w_0\left(\xi + i \frac{\Im(z)}{\cosh(2t)}\right) d\xi + \int_{\gamma_{z, t}} K(t,z, \omega) \, w_0(\omega) d\omega$$
where $\gamma_{z,t}$ is the path formed by the lateral sides of the trapezoid (oriented clockwise). Let us show that the second integral converges to $0$ as $t$ tends to $0$ uniformly on $\overline{D_{L_1}}$. Indeed, using again inequality \eqref{eq-inequality} 
for $z \in D_{L_1}$, we can write
$$ \left|\int_{\gamma_{z,t}} K(t,z, \omega) \, w_0(\omega) d\omega\right| \le \frac{2 \sqrt{2} L_1}{\sqrt{2\pi \sinh(2t)}}  e^{-\frac{1}{2\sinh(2t)} (L_1-L_2)^2} \sup_{\omega \in \overline{D_{L_2}}}|w_0(\omega)|,$$
where $2\sqrt{2}L_1$ is the maximal possible length of $\gamma_{z,t}$.
Finally, we will prove that 
$$t\mapsto w_3(t,z):= \int_{-L_2 + \frac{\Im(z)}{\cosh(2t)}}^{L_2 - \frac{\Im(z)}{\cosh(2t)}}K\left(t,z, \xi + i \frac{\mathrm{Im}(z)}{\cosh(2t)}\right) \, w_0\left(\xi + i \frac{\mathrm{Im}(z)}{\cosh(2t)}\right) d\xi$$
converges to $w_0$ as $t$ tends to $0$ uniformly on $\overline{D_{L_1}}$. We will adapt the proof of \cite[Lemma 3.2]{Dhungana2006} to our complex variable case. Let $\varepsilon >0$. There exists $\delta >0$ such that $L_1 + {\delta}\le L_2$, and  by uniform continuity of $w_0$ on $D_L$, 
\begin{equation}
\label{eq-uniformcontinuity}
\sup_{\substack{z \in \overline{D_{L_1}}\\  \zeta \in \C, \, |z-\zeta| < \delta}} |w_0(z)-w_0(\zeta)|< \varepsilon. 
\end{equation} 
Moreover, note that $\widetilde{K}\left(t,z, \xi + i \frac{\mathrm{Im}(z)}{\cosh(2t)}\right) = \widetilde{K}\left(t,\mathrm{Re}(z), \xi \right)$. 
Hence, writing 
$$w_{0, z, t}(\xi)= w_0\left(\xi + i \frac{\mathrm{Im}(z)}{\cosh(2t)}\right) \mathds{1}_{ \left\lbrace|\xi| + \frac{\left|\mathrm{Im}(z) \right|}{\cosh(2t)}\le L_2 \right\rbrace },  $$ 
and using \eqref{eq-intK1tilde}, we obtain for every $z \in D_{L_1}$  
\begin{align*}
|w_3&(t,z) - w_0(z)| \\
&= \left| \int_{\R}K\left(t,z, \xi + i \frac{\mathrm{Im}(z)}{\cosh(2t)}\right) \, w_{0, z, t}(\xi) d\xi - w_0(z)\right|\\
&= \left|\widetilde{\eta}(t,z) \left(\int_\R  \widetilde{K}\left(t, \mathrm{Re}(z), \xi\right) w_{0, z, t}(\xi) d\xi - w_0(z) \right) + (\widetilde{\eta}(t,z)-1) w_0(z) \right|
\end{align*}
So, by \eqref{eq-intK1tilde} and triangular inequality, we get 
\begin{align*}
|w_3&(t,z) - w_0(z)| \\
&\le |\widetilde{\eta}(t,z)| \int_\R |w_{0, z, t}(\xi) - w_0(z)| \widetilde{K}\left(t, \mathrm{Re}(z), \xi\right) d\xi + |\widetilde{\eta}(t,z)-1| |w_0(z)| \\
& \le |\widetilde{\eta}(t,z)| \sup_{\left|\xi- \frac{\mathrm{Re}(z)}{\cosh(2t)} \right| < \frac{\delta}{2}} \left|w_0\left(\xi + i \frac{\mathrm{Im}(z)}{\cosh(2t)}\right) - w_0(z)\right| \int_{\left|\xi- \frac{\mathrm{Re}(z)}{\cosh(2t)} \right| < \frac{\delta}{2}}  \widetilde{K}(t, \mathrm{Re}(z), \xi) d\xi \\
&\quad + 2 |\widetilde{\eta}(t,z)| \sup_{\zeta \in \overline{D_{L_2}}} |w_0(\zeta)| \int_{\left|\xi- \frac{\mathrm{Re}(z)}{\cosh(2t)} \right| \ge \frac{\delta}{2}}  \widetilde{K}(t, \mathrm{Re}(z), \xi) d\xi 
+ |\widetilde{\eta}(t,z)-1| |w_0(z)|\\
&:= I_1 + I_2 + I_3. 
\end{align*}
Consider $I_1$.
For $t>0$ satisfying $\left(1-\frac{1}{\cosh(2t)}\right)L_1 < \frac{\delta}{4}$, we have 
\begin{align*}
 |\xi+i\frac{\Im (z)}{\cosh(2t)}-z|
 &=|\xi+i\frac{\Im(z)}{\cosh(2t)}-\left[\Re(z)\left(\frac{1}{\cosh(2t)}-\left(\frac{1}{\cosh(2t)}-1\right)\right)+i\Im z|\right]\\
& \le |\xi-\frac{\Re (z)}{\cosh(2t)}|+(|\Re(z)|+|\Im(z)|) \left(1-\frac{1}{\cosh(2t)}\right)\\
&< \frac{\delta}{2}+2\times \frac{\delta}{4}=\delta. 
\end{align*}

Hence, $I_1 \le  C\varepsilon $  by \eqref{eq-uniformcontinuity}, \eqref{eq-intK1tilde}, and $C= e^{L_1^2 \tanh(2t)/2}/\sqrt{\cosh(2t)}$ uniformly bounds $\widetilde{\eta}$ (observe that $C$ goes to 1 when $t\to 0$).

Next, using \eqref{eq-unityapprox}, we deduce that $I_2$ converges to $0$ uniformly on $D_{L_1}$ as $t$ tends to $0$.  Since $\widetilde{\eta}(t,z) \to 1$ uniformly as $t \to 0$ and $w_0$ is bounded on $\overline{D_{L_1}}$, $I_3$ also converges to $0$ as $t$ tends to $0$ uniformly on $\overline{D_{L_1}}$. The proof is complete. 
\end{proof}

We are now in a position to prove Proposition \ref{prop-reach-segment-centered}. 
The following argument is strongly inspired by Strohmaier-Water and exploits the invariance by rotation by $\pm\pi/2$ of the square $D_L$. We would also like to thank S. Ervedoza for
introducing us to this proof and for his version of the said proof. 

\begin{proof}[Proof of Proposition \ref{prop-reach-segment-centered}]
Let $g \in \Hol(D_{L'})$ and let $L<L''<L'$. 
We denote by $\eta \in C^\infty_0(-L', L')$ be a non negative cut-off function such that $\eta = 1$ on $(-L'', L'')$. 
Set
$$
 w_0(x)=
 \begin{cases}
 \eta(x)g(ix),\ x\in (-L', L'),\\
 0,\ x\in\R\setminus (-L', L').
 \end{cases}
$$
Observe that by rotation invariance with center $0$ and angle $\pm \pi/2$ of $D_{L’}$, the function $g(ix)$ is well defined for $x\in (-L', L')$ ($g$ is holomorphic on $D_{L'}\supset i (-L', L')$) and that $w_0=g(i\cdot)$ 
extends holomorphically on $D_{L''}$. 
We denote by $w$ the solution of the equation 
\begin{equation}\label{Eqn17}
\begin{cases}
\partial_t w  -\partial_x^2 w  +x^2 w = 0 , \quad t \in (0, \, \tau)\  x \in \R \\
w(0, x)=w_0(x), \quad x\in \R.
\end{cases}
\end{equation}
Then it is well-known (and it can be deduced from the Mehler kernel)
that $w$ and hence $\partial_t w$ 
 extend to entire functions in the space variable so that replacing $x^2$ by $z^2$ on the left hand side in the first equation of \eqref{Eqn17} yields an entire function vanishing on $\R$. Consequently, by the identity theorem for holomorphic functions and Lemma \ref{lemma-SW}, $w$ is also solution of  
\begin{equation}
\label{eq-extensionw}
\begin{cases}
\partial_t w (t,z)  -\partial_x^2 w(t,z)  +z^2 w(t,z) = 0 , \quad t>0,\  z \in \C \\
w(0, z)=w_0(z)=g(iz), \quad z \in \overline{D_{L}},
\end{cases}
\end{equation}
where, in the second line, we have used the uniform convergence of the solution $w(t,z)$ on compact sets as
$t$ tends to $0$.
Set $v(t,z):=w(t,-iz)$ for all $t >0$ and $z \in \C$. Then we have $\partial_t v (t,z)= \partial_t w (t,-iz)$ and $\partial^2_x v (t,z)=-\partial_x^2 w (t,-iz)$ for every $t>0$ and $z \in \C$. So,
$$\partial_t v (t,z) + \partial^2_x v (t,z) = \partial_t w (t,-iz) - \partial_x^2 w (t,-iz)= -(-iz)^2 w(t,-iz)= z^2 v(t,z).$$
Therefore, in view of \eqref{eq-extensionw}, $v$ satisfies the backward Hermite heat equation
\begin{equation}
\begin{cases}
\partial_t v(t,z)  +\partial_x^2 v(t,z) - z^2 v(t,z) = 0 , \quad t \in (0, \, \tau),\  z \in \C \\
v(0, z)=g(z), \quad z\in \overline{D_{L}}.
\end{cases}
\end{equation}
Using the change of variable $t \to \tau-t$, $\widetilde{v}(t,z)=v(\tau-t,z)$, and turning back to the real variable, this leads to
\begin{equation}
\begin{cases}
\partial_t \tilde{v} (t,x)  -\partial_x^2 \tilde{v}(t,x) + x^2 \tilde{v}(t,x) = 0 , \quad t \in (0, \, \tau),\  x \in I_L \\
\tilde{v}(\tau, x)=g(x), \quad x \in I_L.
\end{cases}
\end{equation}
Finally, $u_{-L}(t):=\tilde{v}(t, -L)= w(\tau- \cdot , iL)$ and $u_L(t):=\tilde{v}(t, L)=w(\tau-t, -iL)$ yield the required boundary $L^2(0, \tau)$-control functions. 
Thus, $g \in \Phi^{\mathrm{H},}_{\tau, (-L, L)}$ and the proof is complete.
\end{proof}

\section{Proof of Theorem \ref{thm-reachable-segment}}\label{S6}

We start with a lemma which is of interest in itself. It translates the following intuitive phenomenon: a solution of the heat equation on a domain $\Omega$ is also solution on each subdomain $O$. Therefore, if the restriction of the solution to the boundary $\partial O$ of the subdomain belongs to $L^2([0,\tau]; L^2(\partial O))$, then the reachable space on $\Omega$ is included in the reachable space on $O$.

Since we are interested in the one-dimensional case, we state the lemma for intervals.

 \begin{lem}
 \label{lemma-reachable-subset}
Let $L \ge \pi$. 
We have the following inclusions of reachable spaces: 
$$\Ran \Phi^{\mathrm{H}}_{\tau, (-L, L)} \subset \Ran \Phi^{\mathrm{H}}_{\tau}$$
and 
$$\Ran \Phi^{\mathrm{H}}_{\tau, 0} \subset \Ran \Phi^{\mathrm{H}}_{\tau}.$$
 \end{lem}
 
\begin{proof}
We only prove the first inclusion, the proof of the second inclusion is analogous (and quite direct from the explicit form of the Mehler kernel).
Let $g \in \Ran \Phi^{\mathrm{H},}_{\tau, (-L, L)}$. Then there exists a control $u:=(u_{-L}, u_L) \in L^2([0, \tau]; \, \C^2)$ such that the solution $w$ of the Hermite heat equation \eqref{eq-Hermiteheat2} satisfies $w(\tau, \, \cdot)=g$. The key point is that we can assume the control to be zero near $t=0$ as noticed in \cite[Proposition 3.2]{KNT2021}. Let us remind the argument. Define 
$$\widetilde{u}_{\pm L}(t):=\begin{cases}
0  &\text{ if } 0<t \le \tau,\\
u_{\pm L}(t-\tau) &\text{ if } \tau<t<2 \tau.
\end{cases}$$
Then the solution $w$ of the equation \eqref{eq-Hermiteheat2} with control $\widetilde{u}:=(\widetilde{u}_{-L}, \widetilde{u}_{L})$ reaches $g$ in time $2 \tau$. 

Now by uniqueness, $w$ is identically zero on $[0, \tau)\times (-L, L)$. Moreover, due to the smoothing effect it is well-known that $w \in C^\infty((0, 2\tau] \times (-L, L))$. Hence, $w$ belongs to $C^\infty([0, 2\tau] \times (-L, L))$ (focus on the regularity at $t=0$).
Finally, it is clear that $w$ satisfies equation \eqref{eq-Hermiteheat} with control $v_0=w(\cdot, 0)$ at $x=0$ and $v_\pi= w(\cdot, \pi)$ at $x=\pi$ which clearly belongs to $L^2(0, 2 \tau)$ by the discussion above. Hence, $g$ belongs to $\Ran \Phi^{\mathrm{H},}_{2 \tau}$ which does not depend on time and the proof is complete.
 \end{proof} 

In the previous proof, we have used the time invariance of the reachable space of the equation on  $(0, \pi)$ to obtain more regularity for the solution at $t=0$. Actually we only need that regularity for the restriction of the solution to an internal point of the interval which is always fulfilled as mentioned briefly in \cite[Remark 2]{SW2020}. 

We also want to notice that the above lemma is probably true for more general linear time invariant parabolic control systems which are null-controllable in any positive time. It turns out to be true for the classical heat equation since in this case we have an exact characterization of all these spaces.

Let us now start the proof of the theorem.


\begin{proof}[Proof of Theorem \ref{thm-reachable-segment}]
Let $\varepsilon >0$. 
Note that $D^{(\varepsilon)}$ is the intersection of the sectors $\Delta:=\enstq{z \in \C}{|\mathrm{arg}(z)| < \pi/4}$ and $\pi + \varepsilon - \Delta$. Hence, we can use a separation of singularities theorem for the Bergman space proved in \cite[Corollary 1.6]{HO2021} 
and which reads as follows: the equality 
$$A^2(D^{(\varepsilon)})=A^2(\Delta) + A^2(\pi + \varepsilon - \Delta)$$
holds. 
Therefore, using equality \eqref{eq-inclusion-heat-halfline}, we get 
\begin{align*}
A^2(D^{(\varepsilon)}) &=  \Ran(\Phi_{T, 0}) + X + A^2(\pi + \varepsilon - \Delta)\\
&\subset \Ran(\Phi_{T, 0}) + \Hol(D_{\pi+\varepsilon}).
\end{align*} 
Now, by Propositon \ref{prop-reach-segment-centered}, $\Hol(D_{\pi+\varepsilon})
\subset \Ran \Phi^H_{\tau,(-\pi,\pi)}$ and by Theorem \ref{thm-reach-half-line},
$\Ran \Phi_{T,0}=\Ran \Phi^{\mathrm{H}}_{\tau, \, 0}$.
%
We conclude the proof with Lemma \ref{lemma-reachable-subset}. 

%
\end{proof}

\subsection*{Acknowledgements}
We would like to thank Sylvain Ervedoza for having presented the Strohmaier-Waters proof to us.

\bibliographystyle{alpha} 
\bibliography{biblio} %

\begin{thebibliography}{MRR16}

\bibitem[AHS90]{AHS1990}
A.~Aikawa, N.~Hayashi, and S.~Saitoh.
\newblock {T}he {B}ergman space on a sector and the heat equation.
\newblock {\em Complex Variables Theory Appl.}, 15:pp. 27--36, 1990.

\bibitem[DE18]{DE18}
J.~Dard{\'e} and S.~Ervedoza.
\newblock On the reachable set for the one-dimensional heat equation.
\newblock {\em SIAM J. Control Optim.}, 3:1692--1715, 2018.

\bibitem[DE19]{DE19}
J\'{e}r\'{e}mi Dard\'{e} and Sylvain Ervedoza.
\newblock Backward uniqueness results for some parabolic equations in an
  infinite rod.
\newblock {\em Math. Control Relat. Fields}, 9(4):673--696, 2019.

\bibitem[Dhu06]{Dhungana2006}
Bishnu~P. Dhungana.
\newblock Mehler kernel approach to tempered distributions.
\newblock {\em Tokyo J. Math.}, 29(2):283--293, 2006.

\bibitem[FR71]{FR1971}
H.~O. Fattorini and D.~L. Russell.
\newblock Exact controllability theorems for linear parabolic equations in one
  space dimension.
\newblock {\em Arch. Rational Mech. Anal.}, 43:272--292, 1971.

\bibitem[HKT20]{HKT2020}
Andreas Hartmann, Karim Kellay, and Marius Tucsnak.
\newblock From the reachable space of the heat equation to {H}ilbert spaces of
  holomorphic functions.
\newblock {\em J. Eur. Math. Soc. (JEMS)}, 22(10):3417--3440, 2020.

\bibitem[HO21]{HO2021}
Andreas Hartmann and Marcu-Antone Orsoni.
\newblock Separation of singularities for the {B}ergman space and application
  to control theory.
\newblock {\em J. Math. Pures Appl. (9)}, 150:181--201, 2021.

\bibitem[KNT21]{KNT2021}
Karim Kellay, Thomas Normand, and Marius Tucsnak.
\newblock Sharp reachability results for the heat equation in one space
  dimension.
\newblock {\em Anal. PDE}, to appear 2021.

\bibitem[LR20]{LR2020}
Camille Laurent and Lionel Rosier.
\newblock Exact controllability of semilinear heat equations in spaces of
  analytic functions.
\newblock {\em Ann. Inst. H. Poincar\'{e} Anal. Non Lin\'{e}aire},
  37(4):1047--1073, 2020.

\bibitem[MRR16]{MRR2016}
P.~Martin, L.~Rosier, and P.~Rouchon.
\newblock On the reachable states for the boundary control of the heat
  equation.
\newblock {\em Appl. Math. Res. Express. AMRX}, pages 81--216, 2016.

\bibitem[Ors21]{Ors2021}
Marcu-Antone Orsoni.
\newblock Reachable states and holomorphic function spaces for the 1-{D} heat
  equation.
\newblock {\em J. Funct. Anal.}, 280(7):108852, 17, 2021.

\bibitem[Sai91]{Saitoh91}
S.~Saitoh.
\newblock Isometrical identities and inverse formulas in the one-dimensional
  heat equation.
\newblock {\em Appl. Anal.}, 40(2-3):139--149, 1991.

\bibitem[Sch86]{Schmidt86}
E.~J. P.~Georg Schmidt.
\newblock Even more states reachable by boundary control for the heat equation.
\newblock {\em SIAM J. Control Optim.}, 24(6):1319--1322, 1986.

\bibitem[Sei79]{Seidman1979}
Thomas~I. Seidman.
\newblock Time-invariance of the reachable set for linear control problems.
\newblock {\em J. Math. Anal. Appl.}, 72(1):17--20, 1979.

\bibitem[SW20]{SW2020}
Alexander Strohmaier and Alden Waters.
\newblock Analytic properties of heat equation solutions and reachable sets,
  2020.

\bibitem[TW09]{TW2009}
M.~Tucsnak and G.~Weiss.
\newblock {\em Observation and Control for Operator Semigroups}.
\newblock Birkh{\"a}user, 2009.

\end{thebibliography}
\end{document}